\documentclass[preprint,12pt]{elsarticle}
\usepackage{latexsym}
\usepackage{amssymb}
\usepackage{amsmath}
\usepackage{amsthm}
\usepackage{verbatim}

\numberwithin{equation}{section}

\newtheorem{theorem}{Theorem}[section]
\newtheorem{lemma}{Lemma}[section]
\newtheorem{proposition}{Proposition}[section]
\newtheorem{remark}{Remark}[section]
\newtheorem{definition}{Definition}[section]

\journal{XXXX}

\begin{document}

\begin{frontmatter}



\title{Existence of multiple periodic solutions for a semilinear wave equation in an $n$-dimensional ball
\tnoteref{ack}
}
\tnotetext[ack]{This work is partially supported by NSFC Grants (nos. 11322105 and 11671071).}


\author{Hui Wei}
\ead{weihui01@163.com}
\author{Shuguan Ji\corref{cor}}
\ead{jishuguan@hotmail.com}
\address{School of Mathematics and Statistics and Center for Mathematics and Interdisciplinary Sciences, Northeast Normal University, Changchun 130024, P.R. China}
\cortext[cor]{Corresponding author.}

\begin{abstract}
This paper is devoted to the study of periodic solutions for a radially symmetric semilinear wave equation in an $n$-dimensional ball.
By combining the variational methods and saddle point reduction technique, we prove there exist at least three
periodic solutions for arbitrary space dimension $n$. The structure of the spectrum of the linearized problem plays an essential role in the proof, and
the construction of a suitable working space is devised to overcome the restriction of space dimension.
\end{abstract}

\begin{keyword}
Existence,  periodic solutions, wave equation
\end{keyword}

\end{frontmatter}


\section{Introduction}

In this paper, we are concerned with the existence of periodic solutions for a radially symmetric semilinear wave equation with periodic-Dirichlet conditions
\begin{eqnarray}
\left\{
\begin{array}{lll}
 u_{tt} - \Delta u = \mu u + f(t, x, u), & t\in \mathbb{R}, \ x \in B_R^n,\\
 u(t,x) = 0, & t\in \mathbb{R}, \ x \in \partial B_R^n, \\
 u(t+T,x) = u(t,x),  & t\in \mathbb{R}, \ x \in B_R^n,
\end{array}
\right.
\label{eqa:1.1}
\end{eqnarray}
where $x=(x_1,\cdots, x_n) \in \mathbb{R}^n$, $\mu>0$  is a constant, $B_R^n = \{ x\in \mathbb{R}^n : \|x\| < R\}$, and  $\partial B_R^n = \{ x\in \mathbb{R}^n : \|x\| = R\}$.

The wave equation is a simplified mathematical model to account for the wave phenomena, such as fluid dynamic, electromagnetic, membrane vibration, etc. In the past few years, many authors (see \cite{Chang.(1981), Chang.(1982), Ding.(1998), Guo.(2004), Ji.(2011), Rabinowitz.(1978), Schechter.(2010), Schechter.(1998), Tanaka.(2006), Wang.(2001)}) payed much attention to the periodic solutions of the wave equation when the space dimension $n=1$. There are also many papers (see \cite{Ben-Naoum.(1995), Ben-Naoum.(1993), Chen.(2017), Chen.(2016), Chen.(2014), Schechter.(1998)}) consider the wave equation in a ball with radius $R$ and time period $T$, but a very interest thing is that the solvability of wave equation in an $n$-dimensional ball with radius $R$ depends on the arithmetical properties of $R$ and $T$. It is well known that, for the one-dimensional case,  if $T=2\pi$ and radius $R= \pi/2$, the structure of the spectral set of the wave operator is made of the eigenvalues $\lambda_{jk}=(2j-1)^2 - k^2$ ($j\in \mathbb{Z_+}=\{1, 2, \cdots\}$, $k\in \mathbb{Z}$). However, for the higher dimensional case, even if $R/T$ are certain rational numbers, the structure of the spectral set of the wave operator becomes more complicated, which makes it more difficult to investigate the periodic solutions of problem \eqref{eqa:1.1}. Ben-Naoum and Mawhin \cite{Ben-Naoum.(1993)} studied the wave equation with a general nonlinear term and obtained at least one $2\pi$-periodic solution in an $n$-dimensional ball with radius $R=\pi/2$, when $n=3$ or $n$ is even. The results are essentially based on the asymptotically behavior of the spectrum of the wave operator.

As regards the multiplicity problem, in \cite{Chen.(2014)}, Chen and Zhang considered the wave equation $ u_{tt} - \Delta u= \mu u + |u|^{p-1}u $ in a ball in $\mathbb{R}^n$
and obtained infinitely many weak solutions for the case that $n-3$ is an integer multiple of $(4,a)$ and $8R/T = a/b$, where $a$, $b$ are relative prime positive integers and $(4,a)$ denotes the greatest common divisor of $4$ and $a$.  Later, they \cite{Chen.(2016)} also dealt with the wave equation $ u_{tt} - \Delta u= \mu u + a(t,x)|u|^{p-1}u $ in a ball with $R= \pi/2$ and obtained infinitely many $2\pi$-periodic solutions for the cases that $n$ is an even integer or $n>3$ is odd. Recently, they \cite{Chen.(2017)} also investigated the wave equation $u_{tt} - \Delta u= g(t,x,u)$ and proved that there exists at least three radially symmetric periodic solutions under some certain suitable conditions for the case that $n=2$ or $n>3$ is odd and $R/T = d/4, d \in \mathbb{Z^+}$.
In this case, it is proved that $0$ is not in the spectral set of the wave operator, which is a crucial fact used in the work.

In this paper, we shall investigate the existence of multiple periodic solutions of problem \eqref{eqa:1.1} for arbitrary space dimension $n$ and $8R/T = a/b$.
By constructing the suitable working space, we can overcome the restriction on space dimension $n$ and prove that the problem \eqref{eqa:1.1} possesses at least three periodic solutions.
Throughout this paper we make the following assumptions:\\
(A1) $f(t, x, u) \in C^1(\mathbb{R}\times B_R^n \times\mathbb{R})$ is radially symmetric with respect to $x$,  $f(t+T, x, u) = f(t, x, u)$, and
\begin{eqnarray}
|f(t, x, u) | = o(|u|), \ \ \textrm{as} \ |u|\rightarrow 0  \ \ \textrm{uniformly in} \ (t, x),
\label{eqa:1.2}
\end{eqnarray}
and $f(t, x, u)$ is asymptotically linear in $u$ at $\infty$ in the following sense: there exists a constant $\beta >0 $ such that
\begin{eqnarray}
|f(t, x, u) - \beta u| = o(|u|), \ \ \textrm{as} \ |u|\rightarrow \infty  \ \ \textrm{uniformly in} \ (t, x).
\label{eqa:1.3}
\end{eqnarray}

The rest of this paper is organized as follows. In Sect. \ref{sec:2}, we give
the definition of weak solution of problem \eqref{eqa:1.1} and transform it into the critical point of corresponding functional.
Meanwhile, we also give the spectral analysis of the wave operator and the statement of the main result
as well as some preliminaries. In Sect. \ref{sec:3}, we reduce the critical point problem of corresponding functional to the finite dimensional
subspace via the saddle point reduction argument. Sect. \ref{sec:4} and Sect. \ref{sec:5} are respectively dedicated to the verification of $(PS)_c$ condition and the bounds of reduction functional.
Finally, in Sect. \ref{sec:6}, we complete the proof of the main result.

\section{Definition of weak solution and some preliminaries}

\setcounter{equation}{0}

\label{sec:2}

By the property of radial symmetry, let $r= \|x\|$, then the periodic solution problem \eqref{eqa:1.1} of $n$-dimensional wave equation can be transformed into
\begin{eqnarray*}
\left\{
\begin{array}{lll}
u_{tt} - u_{rr} - \frac{n-1}{r}u_r= \mu u + f(t, r, u),  & (t, r) \in \Omega,\\
 u(t,R) = 0,  & t\in [0, T], \\
 u(0,r) = u(T,r),  \ \ \ u_t(0,r) = u_t(T,r), & r\in [0, R],
\end{array}
\right.
\label{eqa:2.1}
\end{eqnarray*}
where $\Omega=[0, T] \times [0, R]$ with $R,\, T$ satisfying $8R/T = a/b$ for some relative prime positive integers $a$ and $b$.

Let $D$ denote the class of radially symmetric (in $x$) $T$-periodic (in $t$) functions $\varphi\in C^\infty(\mathbb{R}\times B_R^n )$ which have compact support in  $B_R^n$ for each $t\in \mathbb{R}$.

\begin{definition}
A radially symmetric (in $x$) $T$-periodic (in $t$) function $u$ is called a weak solution of problem \eqref{eqa:1.1} if it satisfies
$$\iint_\Omega \Big( u(\varphi_{tt} - \varphi_{rr}-\frac{n-1}{r}\varphi_r) - (\mu u + f(t, r, u))\varphi\Big) r^{n-1}\textrm{d}t \textrm{d}r = 0,\,\, \forall \varphi \in D. $$
\end{definition}

Denote $\rho = r^{n-1}$ and let
$$L^q(\Omega, \rho) = \Big\{ u: \|u\|^q_{L^q(\Omega, \rho)} = \iint_\Omega |u(t,r)|^q  r^{n-1} \textrm{d}t \textrm{d}r <\infty\Big\},$$
for $q\geq 1$. It is easy to see that $L^2(\Omega, \rho) $ is a Hilbert space equipped with the inner product $\langle u, v \rangle = \iint_\Omega u(t,r)  \overline{v(t,r)} r^{n-1}\textrm{d}t \textrm{d}r$.
We define the linear operator $L_0$ on $L^2(\Omega, \rho)$ by
\begin{eqnarray}
\nonumber
&&L_0 u=\hbar,\ \ \text{iff}\\
\nonumber
 &&\int _{\Omega}u(\varphi _{tt}-\varphi
_{rr}-\frac{n-1}{r}\varphi_r)r^{n-1}\textrm{d}t \textrm{d}r=\int_{\Omega}\hbar\varphi r^{n-1}\textrm{d}t \textrm{d}r,\ \ \forall\varphi\in D.~~~~
\end{eqnarray}
It is known (see, for example, \cite{Ben-Naoum.(1993), Chen.(2017)}) that $L_0$ is a symmetric operator on $L^2(\Omega, \rho)$, and the spectrum of the linear operator $L_0$ is made of eigenvalues
$$\lambda_{jk} = \Big(\frac{\gamma_j}{R}\Big)^2 - \Big(\frac{2k \pi}{T}\Big)^2, \ \quad j\in \mathbb{Z_+}, \quad k\in \mathbb{Z},$$
where $\gamma_j$ is the $j$-th positive zero point of $J_\nu(x)$, $\nu = (n-2)/2$, and $J_\nu(x)$ is the Bessel function of the first kind of order $\nu$.
The corresponding eigenfunctions are
$$\psi_{jk}(t, r) = \frac{1}{R}\sqrt{\frac{2}{T}} \frac{1}{J_{\nu+1}(\gamma_j)} \frac{1}{r^\nu} J_\nu\big( \frac{\gamma_j r}{R} \big)e^{\frac{2k\pi \textrm{i}}{T}t},\ \quad j\in \mathbb{Z_+}, \quad k\in \mathbb{Z}.$$
The orthogonal property of Bessel function
\[ \int^R_0 J_\nu\big( \frac{\gamma_j r}{R}\big) J_\nu\big( \frac{\gamma_k r}{R}\big) r \textrm{d}r=\left \{
\begin{array}{ll}
0, & j\neq k,\\
R^2J_{\nu+1}^2(\gamma_j)/2, & j=k,
\end{array}
\right.
\]
implies that the eigenfunctions $\psi_{jk}$ form a complete orthonormal sequence in $L^2(\Omega, \rho)$.

In the special case that $R = \pi/2$, $T = 2\pi$ and the space dimension $n=1$ or $n=3$, the Bessel function of order $\nu = (n-2)/2$  are $J_{-1/2}(x) = (2x/\pi)^{-1/2}\cos x$ and $J_{1/2}(x) = (2x/\pi)^{-1/2}\sin x$. Hence the spectrum is made of the eigenvalues
\begin{eqnarray*}
\lambda_{jk} = (2j-1)^2 - k^2, \ j\in \mathbb{Z_+}, \ k\in \mathbb{Z}, \ \ \textrm{when} \ n=1,
\end{eqnarray*}
and
\begin{eqnarray*}
\lambda_{jk} = 4j^2 - k^2, \ j\in \mathbb{Z_+}, \ k\in \mathbb{Z}, \ \ \textrm{when} \ n=3,
\end{eqnarray*}
which implies that, in both cases, the eigenvalues are isolated and $0$ is the only eigenvalue of infinite multiplicity. Moreover, Ben-Naoum and Mawhin \cite{Ben-Naoum.(1993)} deduced that, if $R = \pi/2$, $T = 2\pi$ and $n= 2$, then $0$ is not an eigenvalue of $L_0$ and the eigenvalues  of $L_0$ are isolated with finite multiplicity. For more general case that $n$ is positive integers and the ratio $R/T$ is a rational number, Ben-Naoum and Berkovits \cite{Ben-Naoum.(1995)} and Schechter \cite{Schechter.(1998)} obtained the following lemma, which plays an important role in our work.

\begin{lemma}[\cite{Ben-Naoum.(1995), Schechter.(1998)}]
Assume that $8R/T = a/b$, where $a$, $b$ are relatively prime integers. Let $\beta_j = (4j + n -3)\pi /4$, $\tau_k = 2|k|\pi R/T, j\in \mathbb{Z}_+$ and $k\in \mathbb{Z}$, and
we denote by $(4,a)$ the greatest common divisor of $4$ and $a$. Then

{\upshape (i)} $L_0$ has a selfadjoint extension $L$ having no essential spectrum other than the point $\lambda_0=-(n-3)(n-1)/4R^2${\upshape{;}}

{\upshape (ii)} If $n-3$ is not an integer multiple of $(4, a)$, then $L$ has no essential spectrum and $|\beta_j - \tau_k| \geq \pi/4b$ for every $j$, $k${\upshape{;}}

{\upshape (iii)} If $n-3$ is an integer multiple of $(4, a)$, then the essential spectrum of $L$ is precisely the point $\lambda_0= -(n-3)(n-1)/4R^2$. Assume $\lambda$ is in the spectrum of $L$ and $\lambda \notin [2\pi \lambda_0, \lambda_0]$, then $\lambda$ is isolated and the multiplicity of $\lambda$ is finite.

Moreover, for every $j$, $k$, it holds that: either $\beta_j = \tau_k$ or $|\beta_j - \tau_k| \geq \pi/4b$.
When $j$, $k$ satisfy $\beta_j = \tau_k$, the eigenvalues $\lambda_{jk}$ of $L$ accumulate to $\lambda_0= -(n-3)(n-1)/4R^2$, as $j,\ k \rightarrow \infty$; while $\lambda_{j'k'}\rightarrow \infty$ as $j',\ k' \rightarrow \infty$, for $j',\ k'$ satisfying $\beta_{j'} \neq \tau_{k'}$.
\label{lem:2.1}
\end{lemma}

Let $\sigma(L)$ denote the spectrum of $L$. With above lemma in hand, we can define $\beta^+ = \min\{ \lambda \in \sigma(L) :\lambda > \beta\}$ and $\beta^- = \max\{ \lambda \in \sigma(L) :\lambda < \beta\}$, where $\beta$ is present in \eqref{eqa:1.3}. It is obvious that $\beta^- < \beta < \beta^+$.

Now we can state the main result and its proof will be completed in the last section.

\begin{theorem}
Assume that $n$ is an arbitrary positive integer, $8R/T = a/b$ for some relative prime positive integers $a$, $b$,  and $\mu, \ \beta \notin \sigma(L)$ satisfy $\mu \in (0, \beta^+ -\beta)$ and $(\mu, \beta) \cap \sigma(L)\neq \emptyset$. Denote $\mu_0 = \beta^+ -\mu$. If $f$ satisfies (A1) and the following assumption:
\\
(A2)  $f$ is increasing in $u$ and there exists $\eta > 0$ such that
$$\frac{\partial f}{\partial u}(t, x, u) \leq \mu_0 - \eta,  \quad \forall (t, x, u) \in \mathbb{R}\times B_R^n \times\mathbb{R}.$$
Then the problem \eqref{eqa:1.1} has at least three radially symmetric T-periodic solutions.
\label{the:2.2}
\end{theorem}

In the sequel we assume that $n$, $R$, $T$ and $\mu$, $\beta$ satisfy the conditions in Theorem \ref{the:2.2}, except otherwise stated. Since $\mu \notin \sigma(L)$ and $\mu >0$, then  there exists a constant $\delta > 0$ such that
\begin{eqnarray}
|\lambda_{jk} - \mu|\geq \delta >0, \ j\in \mathbb{Z}_+, \ k \in \mathbb{Z}.
\label{eqa:2.2}
\end{eqnarray}

On the other hand, since $\mu \in (0, \beta^+ -\beta)$, then
\begin{eqnarray}
\mu_0 = \beta^+ - \mu  >\beta.
\label{eqa:2.3}
\end{eqnarray}
Moreover, if $\lambda_{jk} > \beta$, we have
\begin{eqnarray}
|\lambda_{jk} - \mu|=\lambda_{jk} - \mu \geq \mu_0 > \beta,  \ j\in \mathbb{Z}_+, \ k\in \mathbb{Z}.
\label{eqa:2.4}
\end{eqnarray}

It is known that, for each $u \in L^2(\Omega, \rho)$, it can be expanded as Fourier series $u(t,r) = \sum\limits_{j\in \mathbb{Z_+}, k\in \mathbb{Z}} \alpha_{jk}(u) \psi_{jk}(t, r)$ with $\alpha_{jk}(u) = \overline{\alpha_{j,-k}(u)} = \langle u, \psi_{jk}\rangle$. We define the working space
$$E= \Big\{u \in L^2(\Omega, \rho):  \|u\|^2_E = \sum\limits_{j\in \mathbb{Z_+}, k\in \mathbb{Z}}|\lambda_{jk} - \mu| |\alpha_{jk}(u)|^2 < \infty\Big\},$$
which is a subspace of $L^2(\Omega, \rho)$. By the estimate  \eqref{eqa:2.2},  $\|\cdot\|_E$ is a norm,  and $E$ is Hilbert space equipped with the inner product $\langle u, v\rangle_0 = \sum\limits_{j, k}|\lambda_{jk} - \mu| \alpha_{jk} \overline{\beta_{jk}}$, where $\alpha_{jk}$ and $\beta_{jk}$ are the Fourier coefficients of $u$ and $v$ respectively. Furthermore, by \eqref{eqa:2.2}, we have
\begin{eqnarray}
\|u\|^2_{L^2(\Omega, \rho)} = \sum\limits_{j, k} |\alpha_{jk}(u)|^2 \leq \delta^{-1} \sum\limits_{j\in \mathbb{Z_+}, k\in \mathbb{Z}}|\lambda_{jk} - \mu| |\alpha_{jk}(u)|^2 = \delta^{-1} \|u\|^2_E,
\label{eqa:2.6}
\end{eqnarray}
which implies that $E$ can be embedded into $L^2(\Omega, \rho)$. Meanwhile, for $u\in E$, the H\"{o}lder inequality and \eqref{eqa:2.6} yield that
\begin{eqnarray}
\|u\|_{L^q(\Omega, \rho)} \leq C \|u\|_E,  \ 1 \leq q \leq 2,
\label{eqa:2.7}
\end{eqnarray}
for some constant $C $ depending on $q$.

Now, we consider the energy functional
\begin{eqnarray}
\Phi(u)  = \frac{1}{2}\langle(L-\mu)u, u\rangle - \iint_\Omega F(t,r, u) r^{n-1}\textrm{d}t \textrm{d}r, \ \ \forall u  \in E,
\label{eqa:2.8}
\end{eqnarray}
where $F(t,r, u) = \int^u_0 f(t,r, s) \textrm{d}s$.
Obviously, $\Phi$ is a $C^1$ functional on $E$, and
\begin{eqnarray}
\langle \Phi'(u), v \rangle = \langle(L-\mu)u, v\rangle - \iint_\Omega f(t,r, u)v r^{n-1}\textrm{d}t \textrm{d}r, \ \ \forall u, \ v\in E.
\label{eqa:2.9}
\end{eqnarray}
Thus $u$ is a weak solution of problem (\ref{eqa:1.1}) if and only if $\Phi'( u) = 0$. Since $f$  is a $C^1$ function, we also have
\begin{eqnarray*}
\langle \Phi''(u)w, v \rangle = \langle (L-\mu)w, v\rangle - \iint_\Omega \frac{\partial f}{\partial u}(t,r, u)v w r^{n-1}\textrm{d}t \textrm{d}r, \ \ \forall u, \ v, \ w \in E.
\end{eqnarray*}
In particular,
\begin{eqnarray}
\langle \Phi''(u)v, v \rangle = \langle (L-\mu)v, v\rangle - \iint_\Omega \frac{\partial f}{\partial u}(t,r, u)v^2 r^{n-1}\textrm{d}t \textrm{d}r, \ \ \forall  u, \ v\in E.
\label{eqa:2.10}
\end{eqnarray}

Thus, the  radially symmetric periodic solutions of problem (\ref{eqa:1.1}) are transformed into the critical points of functional $\Phi$. In what follows,
we will prove the existence of multiple critical points of $\Phi$ by the saddle point reduction technique developed by Amann \cite{Amann.(1979)} and Castro and Lazer \cite{Castro.(1979)}.

\section{The saddle point reduction}

\setcounter{equation}{0}
\label{sec:3}

\begin{lemma}[\cite{Amann.(1979), Castro.(1979)}] Let $H$ be a real Hilbert space with the norm $\|\cdot\|_H$, $\Phi \in C^1(H, \mathbb{R})$, and $H_1$, $H_2$ and $H_3$ be closed subset of $H$ such that $H =H_1\oplus H_2 \oplus H_3$. If there exists a constant $\gamma >0$ satisfying
$$\langle \Phi'( u+w+v_1)-\Phi'( u+w+v_2), v_1 -v_2 \rangle \leq -\gamma \|v_1 -v_2\|^2_H, \ \forall u\in H_2, w\in H_3, v_1, v_2 \in H_1,$$
and
$$\langle \Phi'( u+w_1+v)-\Phi'(  u+w_2+v), w_1 -w_2 \rangle \geq  \gamma \|w_1 -w_2\|^2_H, \ \forall u\in H_2, v\in H_1, w_1, w_2 \in H_3.$$
Then \\
(i) There exists a unique continuous mapping $h: H_2 \rightarrow H_1 \oplus H_3$, such that
$$\Phi(u + h(u)) = \max_{v\in H_1} \min_{w \in H_3} \Phi(u + v+w) = \min_{w \in H_3}\max_{v\in H_1}  \Phi(u + v+w);$$
(ii) Define $\widehat{\Phi}(u) = \Phi(u + h(u))$ for any $u \in H_2$, then $\widehat{\Phi} \in C^1(H_2, \mathbb{R})$, and $\langle \widehat{\Phi}'( u), v \rangle = \langle \Phi'( u +  h(u)), v \rangle, \forall u, v\in H_2$;\\
(iii) If $u \in H_2$ is a critical point of $\widehat{\Phi}$, then $u + h(u)$ is a critical point of $\Phi$. On the other hand, if $u +v$ is a critical point of $\Phi$, where $u \in H_2$, $v \in H_1 \oplus H_3$, then $v= h(u)$, and $u$ is a critical point of $\widehat{\Phi}$;\\
(iv) Furthermore, if $\Phi$ satisfies the Palais-Smale condition $(PS)_c$ at the level $c \in \mathbb{R}$, then the functional $\widehat{\Phi}$ also satisfies the $(PS)_c$ condition.
\label{lem:3.1}
\end{lemma}

 Noting that $\mu, \beta \notin \sigma(L)$, we can decompose $E$ into three orthogonal subspace
\begin{displaymath}
\begin{array}{lll}
E_1= \Big\{u \in E: u = \sum\limits_{\lambda_{jk}<\mu} \alpha_{jk}(u) \psi_{jk}(t, r)\Big\},\\
E_2= \Big\{u \in E: u = \sum\limits_{\mu<\lambda_{jk}<\beta} \alpha_{jk}(u) \psi_{jk}(t, r)\Big\},\\
E_3= \Big\{u \in E: u = \sum\limits_{\lambda_{jk}>\beta} \alpha_{jk}(u) \psi_{jk}(t, r)\Big\}.
\end{array}
\end{displaymath}
Thus we have $E = E_1\oplus E_2 \oplus E_3$. Furthermore, since $(\mu, \beta) \cap \sigma(L) \neq \emptyset$, Lemma \ref{lem:2.1} shows that $E_2 \neq \emptyset$ and $\dim(E_2) < \infty$.

For any $u\in E_1$, which can be expanded as $u = \sum\limits_{\lambda_{j k}<\mu} \alpha_{jk}(u) \psi_{jk}(t, r)$, we have
\begin{equation}
\langle (L - \mu)u, u \rangle= - \sum\limits_{\lambda_{j, k}<\mu} |\lambda_{j k} - \mu| |\alpha_{jk}|^2  = -\|u\|^2_E.
\label{eqa:3.1}
\end{equation}

Similarly, for any $u \in E_2\oplus E_3$, we have
\begin{equation}
\langle (L - \mu)u, u \rangle = \|u\|^2_E.
\label{eqa:3.2}
\end{equation}
Of course, for any $u \in E_2$ or $E_3$, we have $\langle (L - \mu)u, u \rangle = \|u\|^2_E$.

\begin{lemma}
\label{lem:3.2}
If $\mu$, $\beta$ satisfy the conditions in Theorem {\upshape\ref{the:2.2}}, then there exist $\gamma_1>0$, $\gamma_2>0$ such that
\begin{eqnarray}
&&\langle (L - \mu - \beta)u, u \rangle \leq -\gamma_1\|u\|^2_E, \ \forall u\in E_1\oplus E_2, \label{eqa:3.3}\\
&&\langle (L - \mu - \beta)u, u \rangle \geq \gamma_2\|u\|^2_E, \ \forall u\in E_3.\label{eqa:3.4}
\end{eqnarray}
\end{lemma}

\begin{proof}
For  $u \in E_1\oplus E_2$, it can be expanded as $u = \sum\limits_{\lambda_{j k}<\beta} \alpha_{jk}(u) \psi_{jk}(t, r)$,  thus we have
\begin{eqnarray*}
&&\langle (L - \mu -\beta)u, u \rangle \\
&=& \sum\limits_{\lambda_{j k}<\beta} (\lambda_{j k} - \mu - \beta) |\alpha_{jk}|^2\\
&=& \sum\limits_{\lambda_{j k}<\beta} (\lambda_{j k} - \mu ) |\alpha_{jk}|^2 - \beta \sum\limits_{\lambda_{j k}<\beta}  |\alpha_{jk}|^2\\
&\leq& \sum\limits_{\lambda_{j k}<\mu} (\lambda_{j k} - \mu ) |\alpha_{jk}|^2 + \sum\limits_{\mu<\lambda_{j k}<\beta} (\lambda_{j k} - \mu ) |\alpha_{jk}|^2- \beta \sum\limits_{\mu<\lambda_{j k}<\beta}  |\alpha_{jk}|^2.
\end{eqnarray*}
If $\mu<\lambda_{j k}<\beta$, by the definition of $\beta^-$ which is present in Sect. {\ref{sec:2}}, we have $|\lambda_{j k} - \mu|<\beta^-$.
Hence we obtain
\begin{eqnarray*}
&&\langle (L - \mu -\beta)u, u \rangle\\
 &\leq& -\sum\limits_{\lambda_{j k}<\mu} |\lambda_{j k} - \mu | |\alpha_{jk}|^2+ \sum\limits_{\mu<\lambda_{j k}<\beta} |\lambda_{j k} - \mu | |\alpha_{jk}|^2- \frac{\beta}{\beta^-} \sum\limits_{\mu<\lambda_{j k}<\beta}  |\lambda_{j k} - \mu | |\alpha_{jk}|^2\\
 &=& -\sum\limits_{\lambda_{j k}<\mu} |\lambda_{j k} - \mu | |\alpha_{jk}|^2-\left(\frac{\beta}{\beta^-}-1\right)\sum\limits_{\mu<\lambda_{j k}<\beta}  |\lambda_{j k} - \mu | |\alpha_{jk}|^2\\
 &\leq&-\gamma_1\|u\|^2_E,
\end{eqnarray*}
where  $\gamma_1 = \min \{1,  \frac{\beta}{\beta^-} -1\}$, which is positive because of $\beta^- < \beta$.

On the other hand, for $u\in E_3$,  we write $u = \sum\limits_{\lambda_{j k}>\beta} \alpha_{jk}(u) \psi_{jk}(t, r)$. By \eqref{eqa:2.4}, we obtain
\begin{eqnarray*}
\langle (L - \mu -\beta)u, u \rangle &=& \sum\limits_{\lambda_{j k}>\beta} |\lambda_{j k} - \mu | |\alpha_{jk}|^2 - \beta \sum\limits_{\lambda_{j k}>\beta}   |\alpha_{jk}|^2\\
&\geq& \sum\limits_{\lambda_{j k}>\beta} |\lambda_{j k} - \mu | |\alpha_{jk}|^2 - \frac{\beta}{\mu_0} \sum\limits_{\lambda_{j k}>\beta}  |\lambda_{j k} - \mu | |\alpha_{jk}|^2\\
&=& \left(1-\frac{\beta}{\mu_0} \right) \sum\limits_{\lambda_{j k}>\beta} |\lambda_{j k} - \mu | |\alpha_{jk}|^2.
\end{eqnarray*}
Denote $\gamma_2 =  1-\frac{\beta}{\mu_0}$, which is positive by \eqref{eqa:2.3}, then
$$\langle (L - \mu - \beta)u, u \rangle \geq \gamma_2\|u\|^2_E,$$
thus we arrive at the lemma.\qedhere
\end{proof}

Thanks to \eqref{eqa:3.1} and \eqref{eqa:3.2}, we obtain the following lemma which shows that the functional $\Phi$ defined in \eqref{eqa:2.8} satisfies the conditions in Lemma \ref{lem:3.1}.
\begin{lemma}\label{lem:3.3}
If the assumptions in Theorem {\upshape\ref{the:2.2} } hold, then there exists a constant $\gamma > 0$ such that
$$ \langle \Phi'( u+v)-\Phi'( u+w), v -w \rangle \leq -\gamma \|v -w\|^2_E, \ \forall  u\in E_2\oplus E_3, v, w \in E_1,$$
and
$$ \langle \Phi'( u+v)-\Phi'( u+w), v -w \rangle \geq \gamma \|v -w\|^2_E, \ \forall u\in E_1\oplus E_2,  v,  w \in E_3.$$
\end{lemma}
\begin{proof}
For every  $u$, $v$, $w \in E$, we have
\begin{equation}\label{eqa:3.5}
 \langle \Phi'( u+v)-\Phi'( u+w), v -w \rangle
=\int^1_0  \langle \Phi''(u+w+s(v-w))(v-w),  v-w\rangle \textrm{d}s,
\end{equation}
and
\begin{eqnarray}
&&\langle \Phi''(u+w+s(v-w))(v-w),  v-w\rangle \nonumber\\
\label{eqa:3.6}
&=& \langle (L - \mu )(v-w), v-w \rangle - \iint_\Omega (v-w)^2\frac{\partial f}{\partial u} r^{n-1}\textrm{d}t \textrm{d}r.
\end{eqnarray}

On the one hand,  for $v$, $w \in E_1$, $u\in E_2\oplus E_3$, by the assumption (A2) in Theorem \ref{the:2.2} (which implies $\frac{\partial f}{\partial u} \geq 0$) and
\eqref{eqa:3.1}, a direct calculation yields
$$\langle \Phi'( u+v)-\Phi'( u+w), v -w \rangle \leq - \|v -w\|^2_E.$$

On the other hand, for $v$, $w \in E_3$, $u\in E_1\oplus E_2$, by $\eqref{eqa:3.2}$ we have
\begin{eqnarray}
\langle (L - \mu)( v-w), v-w \rangle = \|v-w\|^2_E.
\label{eqa:3.7}
\end{eqnarray}
Moreover, by the assumption $0 \leq \frac{\partial f}{\partial u}(t, x, u) \leq \mu_0 - \eta$ and applying \eqref{eqa:2.4}, we have
\begin{eqnarray}
\iint_\Omega ( v-w)^2\frac{\partial f}{\partial u} r^{n-1}\textrm{d}t \textrm{d}r &\leq& (\mu_0 - \eta ) \|v-w\|^2_{L^2(\Omega, \rho)}\nonumber\\
&\leq& \frac{(\mu_0-\eta)}{\mu_0} \|v-w\|^2_E\nonumber\\
\label{eqa:3.8}
&=&\left(1 - \frac{\eta}{\mu_0}\right) \|v-w\|^2_E.
\end{eqnarray}

By \eqref{eqa:3.5}--\eqref{eqa:3.8}, one follows
$$\langle \Phi'( u+v)-\Phi'( u+w), v -w \rangle \geq \frac{\eta}{\mu_0} \|v-w\|^2_E.$$
Now let $\gamma = \min\{1, \ \frac{\eta}{\mu_0}\}$, thus we obtain the desired results.\qedhere
\end{proof}

By Lemma \ref{lem:3.3} and Lemma \ref{lem:3.1}, for the functional $\Phi$ defined in \eqref{eqa:2.8}, there exists a unique continuous mapping $h : E_2 \rightarrow E_1\oplus E_3$ such that
\begin{equation}
\widehat{\Phi}(u) = \Phi(u + h(u)) = \max_{v\in E_1} \min_{w \in E_3} \Phi(u + v+w)= \min_{w \in E_3}\max_{v\in E_1}  \Phi(u + v+w),
 \label{eqa:3.9}
\end{equation}
and the critical points of functional $\Phi$ on the infinite dimensional space $E$
are equivalent to the critical points of functional  $\widehat{\Phi}$ on the finite dimensional subspace $E_2$.
In what follows, we shall apply the variational method (including the mountain pass lemma (see \cite{Chang.1986})) to obtain  critical points of the functional $\widehat{\Phi}$.

\section{Verification the $(PS)_c$ condition}

\setcounter{equation}{0}
\label{sec:4}

 In order to acquire the critical points of $\widehat{\Phi}$, by Lemma \ref{lem:3.1}, we need to verify that $\Phi$ satisfies $(PS)_c$ condition for any $c \in \mathbb{R}$. That
 means, any sequence $\{u_i\} \subset E$ satisfying $\Phi (u_i) \rightarrow c$ and $\Phi' (u_i) \rightarrow 0$ (as $i \rightarrow \infty$) has a convergent subsequence.

\begin{lemma}\label{lem:4.1} If the assumptions in Theorem {\upshape\ref{the:2.2}} hold. If $\{u_i\} \subset E$ satisfies $\Phi (u_i) \rightarrow c$ and $\Phi' (u_i) \rightarrow 0$ as $i \rightarrow \infty$, then there exists a constant $\widetilde{C} > 0$ independent of $i$  such that $\|u_i\|_E \leq \widetilde{C}$.
\end{lemma}
\begin{proof}
We write $u_i = u^+_i +u^-_i$ with $u^+_i \in E_3$ and $u^-_i \in E_1\oplus E_2$, $i= 1,2, \cdots$.
Firstly, for $u^+_i \in E_3$, by \eqref{eqa:2.9} and $\Phi' (u_i) \rightarrow 0$ (as $i \rightarrow \infty$), we have
\begin{eqnarray}
o(1)\|u^+_i\|_E &\geq & \langle \Phi'( u_i), u^+_i \rangle = \langle (L - \mu)u^+_i, u^+_i \rangle - \iint_\Omega f(t,r, u_i)u^+_i r^{n-1}\textrm{d}t \textrm{d}r \nonumber\\
\label{eqa:4.1R}
&=& \langle (L - \mu - \beta)u^+_i, u^+_i \rangle - \iint_\Omega (f(t,r, u_i)- \beta u_i)u^+_i r^{n-1}\textrm{d}t \textrm{d}r.\ \ \ \ \ \ \
\end{eqnarray}
In virtue of \eqref{eqa:3.4}, we have
\begin{equation}
\label{eqa:4.2R}
\langle (L - \mu - \beta)u^+_i, u^+_i \rangle \geq \gamma_2\|u^+_i\|^2_E.
\end{equation}
In addition, by \eqref{eqa:1.3}, for any $\varepsilon > 0$ small enough, there exists a constant $C =C(\varepsilon) > 0$ such that
\begin{eqnarray}\label{eqa:4.2}
|f(t, x, u_i) - \beta u_i| < \varepsilon |u_i| + C,
\end{eqnarray}
which, combined with \eqref{eqa:2.6} and \eqref{eqa:2.7}, yields that
\begin{eqnarray}
&& \big|\iint_\Omega (f(t,r, u_i)- \beta u_i)u^+_i r^{n-1}\textrm{d}t \textrm{d}r \big| \nonumber\\
&\leq & \varepsilon \|u^+_i\|_{L^2(\Omega, \rho)}\|u_i\|_{L^2(\Omega, \rho)} + C \|u_i^+\|_{L^1(\Omega, \rho)} \nonumber\\
&\leq &\frac{\varepsilon}{2\delta} \|u^+_i\|^2_E + \frac{\varepsilon}{2\delta} \|u_i\|^2_E + C \|u_i^+\|_E,
\label{eqa:4.3}
\end{eqnarray}
for some constant $C$ independent of $i$.
Therefore, by \eqref{eqa:4.1R}, \eqref{eqa:4.2R} and \eqref{eqa:4.3},  we have
\begin{equation}\label{eqa:4.4}
(\gamma_2 - \frac{\varepsilon}{2\delta})\|u^+_i\|^2_E - \frac{\varepsilon}{2\delta}\|u_i\|^2_E - C\|u_i^+\|_E \leq 0.
\end{equation}

Secondly, for $u^-_i \in E_1\oplus E_2$,  we have
\begin{eqnarray*}
o(1)\|u^-_i\|_E &\geq & \langle -\Phi'( u_i), u^-_i \rangle  \\
&=& -\langle (L - \mu - \beta)u^-_i, u^-_i \rangle + \iint_\Omega (f(t,r, u_i)- \beta u_i)u^-_i r^{n-1}\textrm{d}t \textrm{d}r.
\end{eqnarray*}
By a similar calculation in \eqref{eqa:4.3}, we obtain
\begin{eqnarray}
\big|\iint_\Omega (f(t,r, u_i)- \beta u_i)u^-_i r^{n-1}\textrm{d}t \textrm{d}r \big|
\leq \frac{\varepsilon}{2\delta} \|u^-_i\|^2_E + \frac{\varepsilon}{2\delta} \|u_i\|^2_E + C \|u_i^-\|_E,
\label{eqa:4.5}
\end{eqnarray}
for some constant $C$ independent of $i$. Then, in virtue of \eqref{eqa:3.3} and \eqref{eqa:4.5}, one follows
\begin{eqnarray}
(\gamma_1 - \frac{\varepsilon}{2\delta})\|u^-_i\|^2_E - \frac{\varepsilon}{2\delta}\|u_i\|^2_E - C\|u_i^-\|_E \leq 0.
\label{eqa:4.6}
\end{eqnarray}

Finally, noting that $E_3$ and $E_1\oplus E_2$ are orthogonal subspaces of $E$, we have $\|u_i\|^2_E = \|u^+_i\|^2_E + \|u^-_i\|^2_E$. On the other hand, by using the fact that arithmetic mean is no more than quadratic mean for any constants, we have $\|u^+_i\|_E + \|u^-_i\|_E \leq\sqrt{2} \|u_i\|_E$. Therefore,  by denoting $\gamma_0 = \min\{\gamma_1,\gamma_2\}$, the sum of \eqref{eqa:4.4} and \eqref{eqa:4.6} yields
\begin{eqnarray}
(\gamma_0 - \frac{3\varepsilon}{2\delta})\|u_i\|^2_E - C\|u_i\|_E \leq 0.
\label{eqa:4.7}
\end{eqnarray}
Selecting $\varepsilon \in (0, \frac{2\delta \gamma_0}{3})$, then \eqref{eqa:4.7} shows that there exists a constant $\widetilde{C} > 0$ independent of $i$ such that $\|u_i\|_E \leq \widetilde{C}$.
\end{proof}

Now, we rewrite $E = E_1 \oplus E_1^\bot$ for simplicity, where
$$E_1^\bot = E_2\oplus E_3 = \Big\{u \in E: u = \sum\limits_{\lambda_{j k}>\mu} \alpha_{jk}(u) \psi_{jk}(t, r)\Big\}.$$
Denote $E_0$ be the subspace of those $u\in L^2(\Omega, \rho)$ for which $\alpha_{jk}= 0$ if $\beta_j \neq \tau_k$, that is
$$E_0 = \Big\{u \in L^2(\Omega, \rho) : u = \sum\limits_{{j, k}} \alpha_{jk}(u) \psi_{jk}(t, r), \ \beta_j = \tau_k\Big\}.$$
\begin{remark}\label{rem:4.2}
Under the assumptions  of Theorem {\upshape\ref{the:2.2}}, Lemma {\upshape\ref{lem:2.1}} shows that if  $n-3$ is not an integer multiple of $(4, a)$, then $E_0 = \{0\}$. If $n-3$ is an integer multiple of $(4, a)$, then $E_0$ is an infinite dimensional space spanned by the eigenfunctions $\psi_{jk}$ and the corresponding eigenvalues accumulate to $\lambda_0$. Therefore, we have  $\dim (E_1^\bot \cap E_0) <\infty$ for the case $E_0 = \{0\}$ or $\dim(E_0) =\infty$. Moreover, if $\dim(E_0) =\infty$, we have $\dim (E_1 \cap E_0) =\infty$.
\end{remark}
\begin{proposition}[\cite{Chen.(2014)}]\label{pro:4.3}
For all $q \in (1, 2]$, the following embedding
\begin{equation}\label{eqa:4.8}
E\ominus E_0 \hookrightarrow L^q(\Omega, \rho)
\end{equation}
is compact.
\end{proposition}

Since $E$ is a Hilbert space, by Lemma {\upshape\ref{lem:4.1}}, we have $u_i \rightharpoonup u$ weakly as $i \rightarrow \infty$ for some $u \in E$, where $\{u_i\} \subset E$ satisfies $\Phi (u_i) \rightarrow c$ and $\Phi' (u_i) \rightarrow 0$. By the following lemma, we can extract a subsequence of $\{u_i\}$ which converges strongly to some $u \in E$.

\begin{lemma}\label{lem:4.4}
If the assumptions of Theorem {\upshape\ref{the:2.2}} hold, then $\Phi$ satisfies the $(PS)_c$ condition, for any $c \in \mathbb{R}$.
\end{lemma}
\begin{proof}
For any $c \in \mathbb{R}$, assume $\{u_i\} \subset E$ satisfying $\Phi (u_i) \rightarrow c$ and $\Phi' (u_i) \rightarrow 0$ as $i \rightarrow \infty$. We write $u_i = x_i + y_i + w_i + z_i$ and $u = x + y + w + z$,  where $x$, $y$, $w$, $z$ are the weak limits of $\{x_i\}$,  $\{y_i\}$, $\{w_i\}$, $\{z_i\}$ respectively, and $x_i, x \in E_1^\bot \ominus E_0$, $y_i, y \in E_1^\bot \cap E_0$, $w_i, w\in E_1 \ominus E_0$, $z_i, z \in E_1 \cap E_0$.

(i) For $x_i, x \in E_1^\bot \ominus E_0$, we have
\begin{eqnarray*}
\|x_i-x\|^2_E &=& \langle (L - \mu)( x_i-x), x_i-x \rangle \\
&=& \langle (L - \mu) x_i, x_i-x \rangle -\langle (L - \mu) x, x_i-x \rangle.
\end{eqnarray*}
Since $x_i \rightharpoonup x$ weakly in $E_1^\bot \ominus E_0$, we have $\langle (L - \mu) x, x_i-x \rangle \rightarrow 0$ as $i \rightarrow \infty$. Therefore, for $i$ large enough, we have
\begin{equation*}
\|x_i-x\|^2_E \leq \langle (L - \mu) x_i, x_i-x \rangle + o(1).
\end{equation*}

In what follows, we shall prove $\langle (L - \mu) x_i, x_i-x \rangle \rightarrow0$ as $i \rightarrow \infty$.
Noting that $x_i, x \in E_1^\bot \ominus E_0$ and $u_i = x_i + y_i + w_i + z_i$, it is easy to see $u_i - x_i\in (E_1^\bot \ominus E_0)^\perp$. Thus, we have
$$\langle (L - \mu)( u_i-x_i), x_i-x \rangle = 0.$$
Furthermore, by \eqref{eqa:2.9}, one follows
\begin{eqnarray}\label{eqa:4.9}
&&\langle (L - \mu) x_i, x_i-x \rangle = \langle (L - \mu) u_i, x_i-x \rangle \nonumber\\
&=& \langle \Phi'(u_i), x_i-x \rangle + \iint_\Omega f(t,r, u_i)(x_i-x) r^{n-1}\textrm{d}t \textrm{d}r.
\end{eqnarray}
Since $\Phi' (u_i) \rightarrow 0$ as $i \rightarrow \infty$, we have
\begin{eqnarray}\label{eqa:4.10}
\langle \Phi'(u_i), x_i-x \rangle \rightarrow 0, \ \textrm{as} \ \ i \rightarrow \infty.
\end{eqnarray}
In virtue of \eqref{eqa:4.2}, a direct calculation yields
\begin{eqnarray*}
\big|\iint_\Omega f(t,r, u_i)(x_i-x) r^{n-1}\textrm{d}t \textrm{d}r \big|
\leq  (\beta + \varepsilon ) \|u_i\|_{L^2(\Omega, \rho)}\|x_i-x\|_{L^2(\Omega, \rho)}\nonumber
\\ + C\|x_i-x\|_{L^1(\Omega, \rho)}.
\end{eqnarray*}

By Proposition \ref{pro:4.3} and $x_i \rightharpoonup x$ weakly, we obtain the convergence of $x_i \rightarrow x$ strongly in ${L^2(\Omega, \rho)}$ along with a subsequence of $\{x_i\}$. For the sake of convenience, we still use $\{x_i\}$ to denote the subsequence. Moreover, the fact $L^2(\Omega, \rho)  \hookrightarrow L^1(\Omega, \rho)$ shows $x_i \rightarrow x$ strongly in ${L^1(\Omega, \rho)}$. Therefore, we obtain
\begin{equation}\label{eqa:4.11}
\big|\iint_\Omega f(t,r, u_i)(x_i-x) r^{n-1}\textrm{d}t \textrm{d}r \big| \rightarrow 0, \ \textrm{as} \ \ i \rightarrow \infty.
\end{equation}
Inserting \eqref{eqa:4.10}, \eqref{eqa:4.11} into \eqref{eqa:4.9}, one follows
\begin{eqnarray*}
\langle (L - \mu) x_i, x_i-x \rangle \rightarrow 0, \ \textrm{as} \ \ i \rightarrow \infty.
\end{eqnarray*}
Consequently,
\begin{eqnarray}
\|x_i-x\|_E \rightarrow 0,  \ \textrm{as} \ \ i \rightarrow \infty.
\label{eqa:4.12}
\end{eqnarray}

(ii) For $y_i, y \in E_1^\bot \cap E_0$, Remark \ref{rem:4.2} indicates that $\dim (E_1^\bot \cap E_0) <\infty$. Therefore, $y_i \rightharpoonup y$ weakly in $E_1^\bot \cap E_0$ implies
\begin{eqnarray}
\|y_i-y\|_E \rightarrow 0,  \ \textrm{as} \ \ i \rightarrow \infty.
\label{eqa:4.13}
\end{eqnarray}

(iii) For $w_i, w\in E_1 \ominus E_0$, by a similar calculation in $\{x_i\}$, we obtain
\begin{equation}\label{eqa:4.14}
\|w_i-w\|^2_E = - \langle (L - \mu)( w_i-w), w_i-w \rangle \longrightarrow 0, \ as \ \ i \rightarrow \infty.
\end{equation}

(iv)  If $E_0 = \{0\}$, we have $u_i = x_i + w_i$. By \eqref{eqa:4.12} and  \eqref{eqa:4.14}, we arrive at the conclusion.

If $E_0 \neq \{0\}$, it remains to prove that $z_i$ converges strongly to $z$ in $E$ along with a subsequence  of $\{z_i\}$.
Recall $z_i\in E_1 \cap E_0$, then the compact embedding \eqref{eqa:4.8} is invalid for $\{z_i\}$. Thus, we can not extract a strong convergence subsequence of $\{z_i\}$ similar to $\{x_i\}$. On the other hand, since $\dim (E_1 \cap E_0) = \infty$ (see Remark \ref{rem:4.2}), we can not also obtain $\|z_i-z\|_E \rightarrow 0$ as $i \rightarrow \infty$ similar to $\{y_i\}$. In what follows, we will use the monotone method to acquire the desired result.

Since $\Phi' (u_i) \rightarrow 0$ and $z_i \rightharpoonup z$ weakly in $E_1 \cap E_0$, we have
\begin{eqnarray}\label{eqa:4.15}
&&\|z_i-z\|^2_E = - \langle (L - \mu)( z_i-z), z_i-z \rangle \nonumber\\
&=& - \langle \Phi'(u_i), z_i-z \rangle - \iint_\Omega f(t,r, u_i)(z_i-z) r^{n-1}\textrm{d}t \textrm{d}r + \langle (L - \mu) z, z_i-z \rangle\nonumber\\
&\leq &  - \iint_\Omega f(t,r, u_i)(z_i-z) r^{n-1}\textrm{d}t \textrm{d}r + o(1),
\end{eqnarray}
for $i$ large enough.

Denote $ f(u_i)=f(t,r, u_i)$ for convenience. By the definition of the inner product in ${L^2(\Omega, \rho)}$, we have
\begin{eqnarray}\label{eqa:4.16}
&& \iint_\Omega f(t,r, u_i)(z_i-z) r^{n-1}\textrm{d}t \textrm{d}r = \langle f(u_i), z_i-z \rangle  \nonumber\\
&=&\langle f(u_i)- f(\widetilde{u}_i + z), z_i-z \rangle + \langle f(\widetilde{u}_i + z) - f(u), z_i-z \rangle\nonumber\\
 &+& \langle f(u), z_i-z \rangle,
\end{eqnarray}
where $\widetilde{u}_i = x_i + y_i + w_i$.

On the one hand, the assumption (A2) shows that $f$ is increasing in $u$, then
\begin{equation}\label{eqa:4.1a}
\langle f(u_i)- f(\widetilde{u}_i + z), z_i-z \rangle \geq 0.
\end{equation}
Therefore, by \eqref{eqa:4.15}--\eqref{eqa:4.1a}, for $i$ large enough, we have
\begin{equation}\label{eqa:4.2a}
\|z_i-z\|^2_E \leq -\langle f(\widetilde{u}_i + z) - f(u), z_i-z \rangle - \langle f(u), z_i-z \rangle + o(1).
\end{equation}

On the other hand, by \eqref{eqa:4.2}, we have $|f(t, x, u)| < (\varepsilon +\beta) |u| + C$ for given $\varepsilon$, then $f : u \mapsto f(t,r, u)$ is continuous from ${L^2(\Omega, \rho)}$ to ${L^2(\Omega, \rho)}$. By \eqref{eqa:2.6} and  {\eqref{eqa:4.12}--\eqref{eqa:4.14}}, we have $\widetilde{u}_i \rightarrow \widetilde{u}$ strongly  in $L^2(\Omega, \rho)$, where $\widetilde{u} = x + y + w$. Thus, we have
\begin{equation}\label{eqa:4.1b}
\langle f(\widetilde{u}_i + z) - f(u), z_i-z \rangle \rightarrow 0, \ \ as \ \ i \rightarrow \infty.
\end{equation}
Furthermore, since $z_i \rightharpoonup z$ weakly in $E \ominus E_0$, one follows
\begin{equation}\label{eqa:4.1e}
\langle f(u), z_i-z \rangle \rightarrow 0, \ \ as \ \ i \rightarrow \infty.
\end{equation}
Thus, By \eqref{eqa:4.2a}--\eqref{eqa:4.1e}, we have
\begin{eqnarray*}
\|z_i-z\|_E \rightarrow 0,  \ \textrm{as} \ \ i \rightarrow \infty.
\end{eqnarray*}
The proof is completed.
\end{proof}

\section{Bounds of the reduction functional}
\setcounter{equation}{0}
\label{sec:5}

The assertion (i) of the following lemma focuses on the upper bound of the reduction functional $\widehat{\Phi}$. We apply it to acquire one critical point on $E_2$.
The assertion (ii) implies that if $\|u \|_E$ is sufficiently large, then the value of $\widehat{\Phi}$ is no more than $0$. It will be used to obtain the critical point of mountain pass type later.
\begin{lemma}\label{lem:5.1}
If the assumptions of Theorem {\upshape\ref{the:2.2}} hold, then \\
{\upshape (i)} there exists a constant $M > 0$, such that $\widehat{\Phi}(u) < M$, $\forall u \in E_2${\upshape ;}\\
{\upshape (ii)} For $u \in E_2$, there exists a constant $R_0 > 0$, such that $\widehat{\Phi}(u) \leq 0$ for $\|u\|_E \geq R_0$.
\end{lemma}
\begin{proof}
By \eqref{eqa:3.9}, for $u \in E_2$, we have
$$\widehat{\Phi}(u) = \min_{w \in E_3} \max_{v\in E_1} \Phi(u + v+w) \leq \max_{v\in E_1}  \Phi(u + v).$$
The above equality shows that, for the assertion (i), it suffices to prove that $\Phi(u + v)$ has upper bound.

In virtue of \eqref{eqa:2.8}, for any $u\in E_2$, $v\in E_1$, we have
\begin{equation}
\label{eqa:5.1R}
\Phi(u +v)  = \frac{1}{2}\langle(L-\mu - \beta)(u +v), u +v\rangle - \iint_\Omega \left(F(t,r, u +v) - \frac{\beta}{2}(u +v)^2\right)r^{n-1}\textrm{d}t \textrm{d}r.
\end{equation}
By \eqref{eqa:4.2}, it is known that
\begin{eqnarray}
|F(t,r, u +v) - \frac{\beta}{2}(u +v)^2|
\leq   \varepsilon |u +v|^2 + C|u +v|.
\label{eqa:5.1}
\end{eqnarray}
By combining \eqref{eqa:3.3} and \eqref{eqa:5.1},  from \eqref{eqa:5.1R}, it yields
\begin{eqnarray*}
\Phi(u +v) &\leq & -\frac{\gamma_1}{2} \|u +v\|^2_E + \iint_\Omega (\varepsilon |u +v|^2 + C|u +v|) r^{n-1}\textrm{d}t \textrm{d}r \nonumber\\
&\leq & -\frac{\gamma_1}{2} \|u +v\|^2_E  + \varepsilon \|u +v\|^2_{L^2(\Omega, \rho)} + C\|u +v\|_{L^1(\Omega, \rho)},
\end{eqnarray*}
for some constant $C$ depending on $\varepsilon$.

In virtue of \eqref{eqa:2.6}, \eqref{eqa:2.7} and taking $\varepsilon = \frac{\delta \gamma_1}{4}$, the above inequality can be translated into
\begin{eqnarray}
\Phi(u +v) &\leq &  -\frac{\gamma_1}{4} \|u +v\|^2_E + C\|u +v\|_E.
\label{eqa:5.2}
\end{eqnarray}
Therefore, \eqref{eqa:5.2} implies that there exists $M > 0$ such that $\Phi(u +v) \leq M$ for any $u\in E_2$ and $v\in E_1$. The assertion (i)  is established.

In what follows, we prove the assertion (ii). For $u\in E_2$,  $v\in E_1$, the fact $\|u +v\|^2_E = \|u \|^2_E + \|v\|^2_E$ and $\|u +v\|_E \leq \|u \|_E + \|v\|_E$, with the help of \eqref{eqa:5.2}, shows
\begin{eqnarray*}
\Phi(u +v) &\leq & -\frac{\gamma_1}{4} \|u \|^2_E + C \|u \|_E + (-\frac{\gamma_1}{4} \|v \|^2_E + C \|v \|_E)\\
&\leq &  -\frac{\gamma_1}{4} \|u \|^2_E + C \|u \|_E  + C_0,
\end{eqnarray*}
where $C_0 = \max\limits_{s \geq 0} \{-\frac{\gamma_1}{4} s^2 + C s \}$.
Thus, there exists a constant $R_0 > 0$ such that $\Phi(u +v) \leq 0$ for $\|u\|_E \geq R_0$. We arrive at the assertion (ii).
\end{proof}

With the help of the following lemma, we can obtain another critical point on some open ball in $E_2$.
\begin{lemma}\label{lem:5.2}
If the assumptions of Theorem {\upshape\ref{the:2.2}} hold, then for any $\widetilde{R} > 0$ there exists a constant $b$ depending on $\widetilde{R}$ such that $\widehat{\Phi}(u) \geq b$, $\forall  u \in B_{\widetilde{R}} = \{u \in E_2 : \|u \|_E < \widetilde{R}\}$.
\end{lemma}
\begin{proof}
Recalling \eqref{eqa:3.9}, for $u \in E_2$ we have
$$\widehat{\Phi}(u) = \max_{v\in E_1} \min_{w \in E_3} \Phi(u + v+w) \geq \min_{w \in E_3}  \Phi(u + w),$$
where
\begin{eqnarray}
\Phi(u +w)  &=& \frac{1}{2}\langle(L-\mu - \beta)(u +w), u +w\rangle \nonumber
\\ &-& \iint_\Omega \left(F(t,r, u +w) - \frac{\beta}{2}(u +w)^2\right)r^{n-1}\textrm{d}t \textrm{d}r.
\label{eqa:5.3}
\end{eqnarray}

Since $E_2$ and $E_3$ are orthogonal subspaces of $E$, one follows
\begin{equation*}
\langle(L-\mu - \beta)(u +w), u +w\rangle = \langle(L-\mu - \beta)u, u \rangle + \langle(L-\mu - \beta)w, w\rangle.
\end{equation*}
By \eqref{eqa:2.6} and \eqref{eqa:3.2}, we have
\begin{eqnarray*}
&&|\langle(L-\mu - \beta)u, u \rangle| \nonumber\\
&= & | \langle (L-\mu )u, u \rangle - \beta \langle u, u \rangle |
\leq  \|u \|^2_E + \beta  \|u \|^2_{L^2(\Omega, \rho)}
  \leq  (1 + \frac{\beta}{\delta}) \|u \|^2_E = C_1 \|u \|^2_E,
\end{eqnarray*}
where $C_1 = 1 + \frac{\beta}{\delta}$.
Moreover, by \eqref{eqa:3.4}, we have $\langle(L-\mu - \beta)w, w\rangle \geq \gamma_2 \|w \|^2_E$. Therefore, we obtain
\begin{eqnarray}
\langle(L-\mu - \beta)(u +w), u +w\rangle \geq \gamma_2 \|w \|^2_E - C_1 \|u \|^2_E.
\label{eqa:5.4}
\end{eqnarray}

By a similar calculation in \eqref{eqa:5.1}, we also have
\begin{eqnarray}
|F(t,r, u +w) - \frac{\beta}{2}(u +w)^2|\leq \varepsilon |u +w|^2 + C|u +w|,
\label{eqa:5.5}
\end{eqnarray}
for some constant $C$ depending on $\varepsilon$.

Substituting \eqref{eqa:5.4}, \eqref{eqa:5.5} into \eqref{eqa:5.3} and taking $\varepsilon = \frac{\delta \gamma_2}{4}$, by a direct calculation we have
\begin{eqnarray}
\Phi(u +w) &\geq & \frac{\gamma_2}{2} \|w \|^2_E - \frac{C_1}{2} \|u \|^2_E -\frac{\delta \gamma_2}{4} \|u +w\|^2_{L^2(\Omega, \rho)} - C\|u +w\|_{L^1(\Omega, \rho)} \nonumber\\
&\geq & \frac{\gamma_2}{2} \|w \|^2_E - \frac{C_1}{2} \|u \|^2_E - \frac{\gamma_2}{4}(\|u \|^2_E + \|w \|^2_E)- C(\|u \|_E + \|w \|_E)\nonumber\\
&\geq & -(\frac{C_1}{2} + \frac{\gamma_2}{4}) \|u \|^2_E - C \|u \|_E + C_2,
\label{eqa:5.6}
\end{eqnarray}
where $C_2 = \min\limits_{s \geq 0} \{\frac{\gamma_2}{4} s^2 - C s\}$.

For any $\widetilde{R}>0$, let $b = -(\frac{C_1}{2} + \frac{\gamma_2}{4}) \widetilde{R}^2 - C \widetilde{R} + C_2$, then $\widehat{\Phi}(u) \geq b$ for $\|u \|_E < \widetilde{R}$. We complete the proof.
\end{proof}

The following lemma, combined with the assertion (ii) in Lemma \ref{lem:5.1}, helps us to acquire one critical point of mountain pass type on $E_2$ which is different from previous two.
\begin{lemma}\label{lem:5.3}
If the assumptions of Theorem {\upshape\ref{the:2.2}} hold, then there exists a constant $\tau > 0$ and $\bar{r} > 0$ such that $\widehat{\Phi}(u) \geq \tau$, for any $u \in E_2$ with $\|u\|_E = \bar{r}$.
\end{lemma}
\begin{proof}
For any $u \in E_2$, $w \in E_3$,
\begin{eqnarray}
\Phi(u +w)  = \frac{1}{2}\langle(L-\mu )(u +w), u +w\rangle - \iint_\Omega F(t,r, u +w) r^{n-1}\textrm{d}t \textrm{d}r.
\label{eqa:5.7}
\end{eqnarray}
By \eqref{eqa:3.2} and $E_1$, $E_2$ are orthogonal subspaces of $E$, we have
\begin{eqnarray}
\langle(L-\mu )(u +w), u +w\rangle = \|u \|^2_E + \|w \|^2_E.
\label{eqa:5.8}
\end{eqnarray}

Denote $f(\xi) = f(t,r, \xi)$ and $F(\xi) = F(t,r, \xi)$ for convenience. It is easy to see
\begin{eqnarray*}
&&\int^1_0 \int^1_0 s \frac{\partial f}{\partial \xi} (u + s\theta w) w^2 \textrm{d} \theta \textrm{d} s = \int^1_0 w f(u + s w)\textrm{d}s  - f(u)w \\
&=& \int^{u+w}_0 f(s) \textrm{d} s - \int^{u}_0 f(s) \textrm{d} s - f(u)w = F(u + w) - F(u ) -f(u)w.
\end{eqnarray*}
Thus we obtain
\begin{eqnarray}
F(u + w) = \int^1_0 \int^1_0 s \frac{\partial f}{\partial \xi} (u + s\theta w) w^2 \textrm{d} \theta \textrm{d}s + f(u)w + F(u ).
\label{eqa:5.9}
\end{eqnarray}

In what follows, we estimate $\iint_\Omega F(u + w) r^{n-1}\textrm{d}t \textrm{d}r$ by dividing it into three terms:
 $\iint_\Omega  [\int^1_0 \int^1_0 s \frac{\partial f}{\partial \xi} (u + s\theta w) w^2 \textrm{d} \theta \textrm{d}s] r^{n-1}\textrm{d}t \textrm{d}r$, $\iint_\Omega f(u)w r^{n-1}\textrm{d}t \textrm{d}r$ and $\iint_\Omega F(u) r^{n-1}\textrm{d}t \textrm{d}r$.

Firstly, by the assumption (A2), we have
$$\int^1_0 \int^1_0 s \frac{\partial f}{\partial \xi} (u + s\theta w) w^2 \textrm{d} \theta \textrm{d} s \leq \frac{1}{2} w^2 (\mu_0 - \eta).$$
Thus, we obtain
\begin{eqnarray}
\iint_\Omega  \left(\int^1_0 \int^1_0 s \frac{\partial f}{\partial \xi} (u + s\theta w) w^2 \textrm{d} \theta \textrm{d}s \right) r^{n-1}\textrm{d}t \textrm{d}r \leq \frac{\mu_0 - \eta}{2} \|w \|^2_{L^2(\Omega, \rho)}.
\label{eqa:5.10}
\end{eqnarray}
Recalling $w \in E_3$, with the help of \eqref{eqa:2.4}, we obtain $\|w \|^2_{L^2(\Omega, \rho)} \leq \frac{1}{\mu_0} \|w \|^2_E$.
Therefore,
\begin{eqnarray}
\iint_\Omega  \left(\int^1_0 \int^1_0 s \frac{\partial f}{\partial \xi} (u + s\theta w) w^2 \textrm{d} \theta \textrm{d}s \right) r^{n-1}\textrm{d}t \textrm{d}r  \leq \frac{\mu_0 - \eta}{2 \mu_0} \|w \|^2_E.
\label{eqa:5.11}
\end{eqnarray}

Secondly, fix $p > 1$, the assumptions \eqref{eqa:1.2} and \eqref{eqa:1.3} imply that for $\varepsilon >0$ small enough which will be chosen later, then there exists a constant $C=C(\varepsilon) >0$, such that
\begin{eqnarray}
|f(u)| \leq \varepsilon |u| + C |u|^p, \quad \forall \ (t, r, u) \in \Omega \times \mathbb{R}.
\label{eqa:5.12}
\end{eqnarray}
Thus, a direct calculation yields
\begin{eqnarray*}
\Big| \iint_\Omega f(u)w r^{n-1}\textrm{d}t \textrm{d}r \Big|
&\leq& \varepsilon \|u \|_{L^2(\Omega, \rho)} \|w \|_{L^2(\Omega, \rho)} + C \|u \|^p_{L^{2p}(\Omega, \rho)}\|w \|_{L^2(\Omega, \rho)} \nonumber\\
&\leq& \frac{\varepsilon}{2} \|u \|^2_{L^2(\Omega, \rho)}  + C \|u \|^{2p}_{L^{2p}(\Omega, \rho)} + \varepsilon\|w \|^2_{L^2(\Omega, \rho)},
\end{eqnarray*}
where the last inequality is acquired by the Cauchy's inequality with $\varepsilon$.

On the other hand, since $\dim(E_2) < \infty$, then there exists a constant $C>0$ such that $\|u \|^{2p}_{L^{2p}(\Omega, \rho)} \leq C \|u \|^{2p}_E$. Therefore,
\begin{eqnarray}
\Big| \iint_\Omega f(u)w r^{n-1}\textrm{d}t \textrm{d}r \Big| \leq \frac{\varepsilon}{2\delta} \|u \|^2_E  + C \|u \|^{2p}_E + \frac{\varepsilon}{\mu_0}\|w \|^2_E,
\label{eqa:5.13}
\end{eqnarray}
for some constant $C$ depending on $\varepsilon$ and $p$.

Thirdly, by \eqref{eqa:5.12}, we have $|F(u)| \leq \frac{\varepsilon}{2} |u|^2 + C |u|^{p+1}$. Since $\dim(E_2) < \infty$, similarly, there exists a constant $C>0$ such that $\|u \|^{p+1}_{L^{p+1}(\Omega, \rho)} \leq C \|u \|^{p+1}_E$. Therefore,
\begin{eqnarray}
\Big|\iint_\Omega F(u) r^{n-1}\textrm{d}t \textrm{d}r \Big| &\leq & \frac{\varepsilon}{2} \|u \|^2_{L^2(\Omega, \rho)} + C \|u \|^{p+1}_{L^{p+1}(\Omega, \rho)} \nonumber \\
&\leq & \frac{\varepsilon}{2\delta} \|u \|^2_E + C \|u \|^{p+1}_E,
\label{eqa:5.14}
\end{eqnarray}
for some constant $C$ depending on $\varepsilon$ and $p$.

Consequently, the sum of \eqref{eqa:5.11}, \eqref{eqa:5.13} and \eqref{eqa:5.14} yields
\begin{eqnarray}\label{eqa:5.15}
\Big|\iint_\Omega F(u + w) r^{n-1}\textrm{d}t \textrm{d}r \Big| \leq \frac{\varepsilon}{\delta} \|u \|^2_E + C \|u \|^{p+1}_E + C \|u \|^{2p}_E \nonumber \\
+ \frac{\mu_0 - \eta + 2\varepsilon}{2 \mu_0} \|w \|^2_E.
\end{eqnarray}

Finally, substituting \eqref{eqa:5.8} and \eqref{eqa:5.15} into \eqref{eqa:5.7}, one follows
\begin{eqnarray}\label{eqa:5.16}
\Phi(u +w) &\geq& (\frac{1}{2} - \frac{\varepsilon}{\delta}) \|u \|^2_E - C \|u \|^{p+1}_E - C \|u \|^{2p}_E + (\frac{1}{2} - \frac{\mu_0 - \eta + 2\varepsilon}{2 \mu_0}) \|w \|^2_E \nonumber\\
&=& (\frac{1}{2} - \frac{\varepsilon}{\delta}) \|u \|^2_E  - C \|u \|^{p+1}_E - C \|u \|^{2p}_E + ( \frac{ \eta - 2\varepsilon}{2 \mu_0}) \|w \|^2_E.
\end{eqnarray}
Taking $\varepsilon = \min\{\frac{\delta}{4}, \frac{\eta}{4}\}$ in \eqref{eqa:5.16}, we have
\begin{eqnarray*}
\Phi(u +w) \geq  \frac{1}{4}\|u \|^2_E - C \|u \|^{p+1}_E - C \|u \|^{2p}_E.
\end{eqnarray*}

Now we consider the function
\begin{eqnarray*}
\phi (s) = \frac{1}{4}s^2 - C s^{p+1} - C s^{2p}, \ \forall \ s \geq 0.
\end{eqnarray*}
Since  $p > 1$, it is easy to see that $\phi$ attains local minimum at $s = 0$. Therefore there exist  two constants $\bar{r}>0$ and $\tau > 0$ such that $\phi (\bar{r}) \geq \tau$.
Recalling  $\widehat{\Phi}(u)  \geq \min\limits_{w \in E_3}  \Phi(u + w)$, which leads to
$\widehat{\Phi}(u)  \geq \tau $, for $ u \in E_2$  satisfying $\|u \|_E = \bar{r}$.
\end{proof}

\section{Proof of Theorem \ref{the:2.2}}

\setcounter{equation}{0}
\label{sec:6}

In this section, we shall give the proof  of Theorem {\upshape\ref{the:2.2}}.
\begin{proof}
Firstly, we assert that there exist a local minimum point and a global maximum point in $E_2$.

On the one hand, the assumption \eqref{eqa:1.2} and (A2) imply that $f(t, r, 0)=0$ and  $f(t, r, \xi)\geq 0$  for $\xi \geq 0$. Therefore,  we have $F(t, r, u)\geq 0$ for any $u \in E$.  For $v \in E_1$, by \eqref{eqa:3.1}, we obtain
$$\Phi(v)  = \frac{1}{2}\langle(L-\mu)v, v\rangle - \iint_\Omega F(t,r, v) r^{n-1}\textrm{d}t \textrm{d}r\leq 0.$$
Therefore,
\begin{equation}\label{eqa:6.1R}
\widehat{\Phi}(0) = \min_{w \in E_3} \max_{v\in E_1} \Phi(v+w) \leq \max_{v\in E_1}  \Phi(v) \leq 0.
\end{equation}

By Lemma \ref{lem:5.2}, Lemma \ref{lem:5.3}, and $\widehat{\Phi}(0) \leq 0$, taking $\widetilde{R} = \bar{r}$ and noting $0 \in B_{\bar{r}} = \{u \in E_2 : \|u \|_E < \bar{r}\}$, then the reduction functional  $\widehat{\Phi}$ attains its infimum in  $B_{\bar{r}}$. Let $\sigma_1 = \inf\limits_{u \in B_{\bar{r}}}\widehat{\Phi}(u)$, by Lemma \ref{lem:3.1} and  Lemma \ref{lem:4.4}, we have  that $\widehat{\Phi}$ satisfies the $(PS)_{\sigma_1}$ condition. Therefore, there exists $u_1 \in  E_2$ such that $\widehat{\Phi}' (u_1)= 0$ and $\widehat{\Phi} (u_1)= \sigma_1$.

On the other hand, let $\sigma_2 = \sup\limits_{u \in E_2}\widehat{\Phi}(u)$, by Lemma \ref{lem:4.4} and Lemma \ref{lem:5.1}, one follows that $\widehat{\Phi} $ satisfies the $(PS)_{\sigma_2}$ condition and it has upper bound respectively. Similarly, there exists $u_2 \in  E_2$ such that $\widehat{\Phi}' (u_2)= 0$ and $\widehat{\Phi} (u_2)= \sigma_2$.

Now we prove $u_1$ and $u_2$ are two different points in $E_2$.  By \eqref{eqa:6.1R}, Lemma \ref{lem:5.3} and $0 \in B_{\bar{r}}$, we have
\begin{eqnarray}
\inf\limits_{u \in B_{\bar{r}}}\widehat{\Phi}(u) \leq \widehat{\Phi}(0) \leq \max_{v\in E_1}  \Phi(v) \leq 0 < \tau \leq \inf\limits_{\|u\|=\bar{r}}\widehat{\Phi}(u)\leq \sup\limits_{u \in E_2}\widehat{\Phi} (u),
\label{eqa:5.17}
\end{eqnarray}
thus $u_1 \neq u_2$.

In what follows, we prove that there exists a third critical point distinct from $u_1$ and $u_2$. We divide the proof into the following two cases.

\emph{Case 1}: If $\widehat{\Phi}$ has another local maximum point which is different from $u_2$, then there exist at least three critical points of $\widehat{\Phi}$.

\emph{Case 2}: If $u_2$ is the unique  maximum point of $\widehat{\Phi}$. Taking $u_0 \in E_2$ with $\|u_0\|_E = 1$, by Lemma \ref{lem:5.1}, there exists $R_0 > \bar{r}$ and $\widehat{\Phi} (R_0 u_0) \leq 0$. Moreover, one of the facts holds: either $sR_0 u_0 \neq u_2$ or $-sR_0 u_0 \neq u_2$ for all $s \in [0, 1]$.

(i) If $sR_0 u_0 \neq u_2$ for all $s \in [0, 1]$, Lemma \ref{lem:5.3} and inequality \eqref{eqa:5.17} imply that
$$\max \{\widehat{\Phi} (0), \widehat{\Phi} (R_0 u_0)\} \leq 0 < \tau \leq \inf\limits_{\|u\|= \bar{r}} \widehat{\Phi} (u).$$

Let
$$c^+ = \inf\limits_{g \in \Sigma^+} \max\limits_{s \in [0, 1]} \widehat{\Phi} (g(s)),$$
 where $\Sigma^+ = \{g \in C([0, 1], E_2) : g(0)=0, g(1)=R_0 u_0\}$. Recalling that $\widehat{\Phi}$ is $C^1$ function and satisfies $(PS)_{c^+}$ condition. By the mountain pass lemma, we obtain that  $c^+$ is a critical value of $\widehat{\Phi}$ and satisfies $c^+ \geq \tau >0$. Therefore, there exists $u_3 \in E_2$ such that $\widehat{\Phi} (u^+_3) = c^+$ and $\widehat{\Phi}' (u^+_3) = 0$.

 Furthermore, for all $s \in [0, 1]$,  since $g_0(s) = sR_0 u_0 \in \Sigma^+$ is not the  maximum point of $\widehat{\Phi}$, then
 $$c^+ \leq \max\limits_{s \in [0, 1]} \widehat{\Phi} (g_0(s)) < \sup\limits_{u\in E_2}\widehat{\Phi}(u)= \sigma_2.$$
Thus
$$\sigma_1 = \inf\limits_{u \in B_{\bar{r}}}\widehat{\Phi}(u)\leq 0< \tau \leq c^+ \leq \max\limits_{s \in [0, 1]} \widehat{\Phi} (g_0(s)) < \sup\limits_{u\in E_2}\widehat{\Phi}(u) =\sigma_2$$
implies that $u_1$, $u_2$ and $u_3$ are three different critical points.

(ii) If $-sR_0 u_0 \neq u_2$ for all $s \in [0, 1]$, similarly, we have
$$c^- = \inf\limits_{g \in \Sigma^-} \max\limits_{s \in [0, 1]} \widehat{\Phi} (g(s))$$
is the critical value of $\widehat{\Phi}$, where
 $\Sigma^- = \{g \in C([0, 1], E_2) : g(0)=0, g(1)=-R_0 u_0\}$. Therefore, there exists $u^-_3 \in E_2$ such that $\widehat{\Phi} (u^-_3) = c^-$ and $\widehat{\Phi}' (u^-_3) = 0$. Moreover, we have $\sigma_1 < c^-  < \sigma_2$.
 This implies the reduction function $\widehat{\Phi}$ has three critical points.

Therefore, the the reduction function $\widehat{\Phi}$ has at least three critical points whenever either $sR_0 u_0 \neq u_2$ or $-sR_0 u_0 \neq u_2$ holds. Consequently, by Lemma \ref{lem:3.1}, it follows that the energy function $\Phi$ has at least three critical points. We complete the proof.
\end{proof}















\bibliographystyle{elsarticle-num}
\bibliography{<your-bib-database>}




\end{document}